\documentclass{amsart}

\DeclareMathOperator{\conv}{conv}
\DeclareMathOperator{\lin}{lin}

\DeclareMathOperator{\id}{Id}
\newtheorem{twr}{Theorem}[section]
\newtheorem{lem}[twr]{Lemma}

\newtheorem{cor}[twr]{Corollary}
\newtheorem{problem}[twr]{Problem}
\theoremstyle{remark}

\numberwithin{equation}{section}

\begin{document}
\pagestyle{plain}
\title{Hyperplanes of finite-dimensional normed spaces with the maximal relative projection constant}

\author{Tomasz Kobos}


\begin{abstract}
The \emph{relative projection constant} $\lambda(Y, X)$ of normed spaces $Y \subset X$ is defined as $\lambda(Y, X) = \inf \{ ||P|| : P \in \mathcal{P}(X, Y) \}$, where $\mathcal{P}(X, Y)$ denotes the set of all continuous projections from $X$ onto $Y$. By the well-known result of Bohnenblust for every $n$-dimensional normed space $X$ and its subspace $Y$ of codimension $1$ the inequality $\lambda(Y, X) \leq 2 - \frac{2}{n}$ holds. The main goal of the paper is to study the equality case in the theorem of Bohnenblust. We establish an equivalent condition for the equality $\lambda(Y, X) = 2 - \frac{2}{n}$ and present several applications. We prove that every three-dimensional space has a subspace with the projection constant less than $\frac{4}{3} - 0.0007$. This gives a non-trivial upper bound in the problem posed by Bosznay and Garay. In the general case, we give an upper bound for the number of $(n-1)$-dimensional subspaces with the maximal relative projection constant in terms of the facets of the unit ball of $X$. As a consequence, every $n$-dimensional normed space $X$ has an $(n-1)$-dimensional subspace $Y$ with $\lambda(Y, X) < 2-\frac{2}{n}$. This contrasts with the seperable case in which it is possible that every hyperplane has a maximal possible projection constant.
\end{abstract}

\maketitle

\section{Introduction}

Let $X$ be a real Banach space and $Y$ its closed subspace. A linear bounded operator $P:X \to Y$ is called a \emph{projection} if $P|_{Y}=\id_{Y} $. By $\mathcal{P}(X, Y)$ we denote the set of all projections from $X$ onto $Y$. The \emph{relative projection constant} of $Y$ is defined as
$$\lambda(Y, X) = \inf \{ ||P|| : P \in \mathcal{P}(X, Y) \}.$$
Moreover, if a projection $P: X \to Y$ satisfies $||P|| = \lambda(Y, X)$ then $P$ is called a \emph{minimal projection}.

Minimal projections have gained a considerable attention in the past years. Many authors have studied their properties in the context of functional analysis and approximation theory (see for example \cite{lewickichalm},  \cite{lewickichalm2}, \cite{cheneyinni}, \cite{cheneymorris}, \cite{cheneyfranchetti}, \cite{konig}, \cite{konig2}, \cite{lewicki}). Some of the obtained results are concerned with studying minimal projections in certain classical Banach spaces and some of them are of more general nature. Results provided in this paper belong to the second class. Our goal is to investigate some general properties of minimal projections in the setting of finite dimensional real normed spaces.

The problem of giving the upper bound for the relative projection constant in the case of an arbitrary subspace has already been studied quite intensively. One of the most fundamental results in this category is an old theorem of Bohnenblust on projections onto subspaces of codimension $1$ of finite dimensional real normed spaces.

\begin{twr}[Bohnenblust \cite{bohnenblust}]
\label{twbohnenblust}
Let $X$ be a real $n$-dimensional Banach space and let $Y \subset X$ be its $(n-1)$-dimensional subspace. Then $\lambda(Y, X) \leq 2 - \frac{2}{n}$.
\end{twr}

One can easily see that this estimation is optimal: if $X = \ell_{1}^n$ or $X = \ell_{\infty}^n$ and $Y = \ker f$, where $f(x) = x_1 + x_2 + \ldots + x_n$, then $\lambda(Y, X) = 2 - \frac{2}{n}$ (see Theorem \ref{projekcjal1}).

In the context of an arbitrary subspace we have the Kadec-Snobar Theorem:
\begin{twr}[Kadec, Snobar \cite{kadecsnobar}]
Let $X$ be a real $n$-dimensional Banach space and let $Y \subset X$ be its $k$-dimensional subspace. Then $\lambda(Y, X) \leq \min\{ \sqrt{k}, \sqrt{n-k} + 1 \}$.
\end{twr}
This estimation was further improved by several authors, see for example \cite{makai}.

Much less research has been done in the problem of \text{finding} a subspace with small projection constant in an arbitrary normed space. There is an old and still unanswered question of Bosznay and Garay.

\begin{problem}[Bosznay, Garay \cite{bosznay}]
\label{bosz}
For an integer $n \geq 3$ determine the value of $\sup_{X}  \inf_{Y \subset X} \lambda(Y, X)$, where $X$ is a real $n$-dimensional normed space and $Y \subset X$ is a subspace of dimension at least $2$ and at most $n-1$.
\end{problem}
To our knowledge, the best known estimates in the Problem \ref{bosz} are these which hold for an arbitrary subspace $Y$. Even in the three-dimensional setting there seem to be a lack of any better bounds. The aim of the paper is to give some result in this direction.

We shall consider the case of the projections onto hyperplanes. To shed some light onto the question of Bosznay and Garay (and similar ones) we bring some attention to studying the equality case in the theorem of Bohnenblust. We provide the following characterization.

\begin{twr}
\label{warunek}
Let $X$ be an $n$-dimensional normed space and let $Y = \ker f$, where $f \in S_{X^{\star}}$, be an $(n-1)$-dimensional subspace of $X$. Then $\lambda(Y, X) = 2 - \frac{2}{n}$ if and only if there exist extreme points $x_1, x_2, \ldots, x_n$ of the unit ball of $X$ such that the following conditions are satisfied
\begin{itemize}
\item $f(x_1)=f(x_2) = \ldots = f(x_n),$
\item vectors $x_1, x_2, \ldots, x_n$ are linearly independent,
\item if an arbitrary vector $\mathbb{R}^n \ni x = \sum_{i=1}^{n} w_i x_i $ is written in the basis of $x_i$, the folowing inequality holds
$$\max_{i=1, 2, \ldots, n} \{ |w_1 + w_2 + \ldots + w_{i-1} - w_i + w_{i+1} + \ldots + w_n | \} \leq ||x||.$$
\end{itemize} 

The third condition is equivalent to the fact that for every $1 \leq i \leq n$ the set
$$\{ x_1, x_2, \ldots, x_{i-1}, -x_i, x_{i+1}, \ldots, x_n\}$$
is contained in a facet of the unit ball.
\end{twr}

Proof of this theorem is provided in Section $2$. This equivalent condition has several consequences. Those of them, which hold in an arbitrary dimension are discussed further in Section $2$. For instance, we can easily obtain a upper bound for the number of hyperplanes with the maximal relative projection constant in terms of the number of facets of the unit ball (see Theorem \ref{sciany}). As a consequence, every $n$-dimensional subspace $X$ has an $(n-1)$-dimensional subspace $Y$ with $\lambda(Y, X) < 2 -\frac{2}{n}$. This is finite-dimensional phenomen as in the seperable case the situation can be different. We also provide purely geometric characterization of the equality $\lambda(Y, X) = 2 - \frac{2}{n}$ (see Corollary \ref{charakteryzacjan}). As an other application of Theorem \ref{warunek} we observe that every $n$-dimensional normed space $X$ which has an $(n-1)$-dimensional subspace with the maximal possible relative projection constant has also a two-dimensional subspace with \textbf{minimal} possible relative projection constant (equal to $1$) (see Corollary \ref{maxmin}).

In the Section $3$ we take a closer look a the three-dimensional case, in which something more can be said. In this setting, the condition $\lambda(Y, X) = \frac{4}{3}$ seems to be much more restrictive on the unit ball of $X$ than in the general case. This allows us to strengthen some results obtained in the preceeding section. In particular, we prove that the maximal possible number of subspaces $Y$ for which the equality $\lambda(Y, X) = \frac{4}{3}$ holds is equal to $4$ (see Theorem \ref{plaszczyzny}). Moreover, from Corollary \ref{maxmin} it follows that every three-dimensional normed space $X$, which posess a subspace $Y$ with $\lambda(Y, X) = \frac{4}{3}$ posess also a subspace $Z$ satisfying $\lambda(Y, Z)=1$. In Theorem \ref{maxmin3} we provide a stability version of this result, which gives some improvement in the three-dimensional case of Problem \ref{bosz} (see Corollary \ref{bosz3}). We note that the improvement is very small, but still, to our knowledge, it is the first non-trivial estimate in this direction. We suspect that the actual constant is much smaller than given in our corollary. While we are not aware of any results concerning Problem \ref{bosz}, we should mention the related papers \cite{franchetti2}, \cite{franchetti1} of Franchetti. Among other things, Franchetti have studied in them the connection between the relative projection constant $\lambda(Y, X)$ (where $Y$ is a hyperplane in not necessarily finite-dimensional Banach space $X$) and behaviour of the norm in the hyperplanes parallel to $Y$. Such a behaviour plays also a major role in the proof of our main Theorem \ref{warunek}. Neverthless, there does not seem to be overlap between our results and the results of Franchetti.

In the last section of the paper we propose some naturally arising questions which are suitable for further research.

\section{The general case}
Troughout the paper we shall always consider only real $n$-dimensional normed spaces $X$ with $n \geq 3$. The unit ball and the unit sphere of such a normed space $X$ will be denoted by $B_X$ and $S_X$ respectively. Let us also recall that a \emph{face} of a convex body $C \subset \mathbb{R}^n$ is the intersection of it with some supporting hyperplane. A face is called a \emph{facet} if it is $(n-1)$-dimensional, or in other words, it is not contained in an affine subspace of dimension $n-2$. The vectors from the canonical unit basis of $\mathbb{R}^n$ will be denoted by $e_1, e_2, \ldots, e_n$. By $\ell_1^n$ and $\ell_{\infty}^n$ we denote the space $\mathbb{R}^n$ equipped with the norm $||x||_1 = |x_1| + |x_2| + \ldots + |x_n|$ and $||x||_{\infty} = \max_{1 \leq i \leq n} |x_i|$ respectively. We shall often use a simple fact that if $X$ is an normed space and $Y = \ker f \subset X$ (where $f \in S_{X^{\star}})$ is a subspace of codimension $1$ then every projection $P:X \to Y$ can be written in the form $P(x) = x - f(x)r$ for some $r \in \mathbb{R}^n$ satisfying $f(r)=1$.

We begin with a lemma used already in the original paper of Bohnenblust. The result is well-known and often used in the study of minimal projections, but we provide its short proof.
\begin{lem}
\label{helly}
Let $X$ be an $n$-dimensional normed space and let $Y$ be its $(n-1)$-dimensional subspace. Suppose that for every extreme points $x_1,\ x_2,\ \ldots,\ x_n$ of the unit ball of $X$ there exists a projection $P:X \to Y$ such that $||P(x_i)|| \leq m$ for every $i=1, 2, \ldots, n$ and some positive real number $m>0$. Then there exists a projection $P:X \to Y$ such that $||P|| \leq m$.
\end{lem}
\begin{proof}
For an arbitrary extreme point $x_0 \not \in Y$ of the unit ball of $X$ let us denote by $\mathcal{P}_{x_0}$ the set of all projections $P: X \to Y$ such that $||P(x_0)|| \leq m$. It is not hard to verify that the set of all projections $\mathcal{P}(X, Y)$ forms an $(n-1)$-dimensional space and the set $\mathcal{P}_{x_0}$ is a compact and convex set of this space. According to our assumption, the intersection of any $n$ of the sets $\mathcal{P}_{x_0}$ is non-empty. From Helly's Theorem it follows that the intersection all sets of the form $\mathcal{P}_{x_0}$ is non-empty. Therefore, there exists a projection $P: X \to Y$ such that $||P(x_0)|| \leq m$ for an arbitrary extreme point $x_0$ of the unit ball. But the unit ball of $X$ is the convex hull of its extreme points and therefore $||P(x_0)|| \leq m$ for an arbitrary $x_0 \in B_X$. This concludes the proof.
\end{proof}

In the proof of Theorem \ref{twbohnenblust} Bohnenblust have managed to reduce the case of the general normed space to the case of the space $\ell_1^n$. One can therefore expect that it may be possible to use some more advanced results concerning the $\ell_1^n$ space in studying the relative projection constant of the hyperplanes. The result we refer to is the explicit formula for the relative projection constant of the hyperplane in $\ell_1^n$. It is quite complicated in the general case, however we need only some simpler consequences of it, given in Lemmas \ref{l1} and \ref{l12}. 

\begin{twr}
\label{projekcjal1}
Let $Y=\ker f$ be an $(n-1)$-dimensional subspace of the space $\ell_1^n$ where $n \geq 3$. Suppose that functional $f$ is given by the vector $(f_1, f_2, \ldots, f_n)$ where $1 = f_1 \geq f_2 \geq \ldots \geq f_n \geq 0$. Let $1 \leq k \leq n$ be the largest integer such that $f_k>0$. Let $a_i = \sum_{j=1}^{i} f_j$, $ b_j=\sum_{j=1}^{i} f_j^{-1}$ for $1 \leq i \leq k$ and $\beta_i = \frac{b_i}{i-2}$ for $3 \leq i \leq k$. Let $3 \leq l \leq k$ be the largest integer such that both of the numbers $f_lb_{l-1}$ and $a_{l-1}$ are greater than $l-3$. Then $\lambda(Y, \ell_1^{n})=1+x$, where
\[x =
\begin{cases}
0 &\mbox{ if } k \leq 2\\ 
2 \left ( \left ( \beta_l - f_l^{-1} \right )(l-2) + a_lf_l^{-1} - l \right )^{-1} &\mbox{ if } k>2 \text{ and } a_l < l-2\\
2\left ( a_l \beta_l - l \right )^{-1} &\mbox{ if } k>2 \text{ and } a_l \geq l-2. \\
\end{cases}
\]
\end{twr}
\begin{proof}
See Theorem 2.2.13 in \cite{lewickihab} on page 57.
\end{proof}

\begin{lem}
\label{l1proj} 
Let $Y=\ker f$ be an $(n-1)$-dimensional subspace of the space $\ell_1^n$ where $n \geq 3$. Suppose that functional $f \neq 0$ is given by the vector $(f_1, f_2, \ldots, f_n)$. Then $\lambda(Y, \ell_1^{n}) = 2 - \frac{2}{n}$ if and only if $|f_1|=|f_2| = \ldots = |f_n|$.
\end{lem}
\begin{proof}
Without loss of generality we may assume that $1 = f_1 \geq f_2 \geq \ldots \geq f_n \geq 0$. Suppose that $Y=\ker f$ is an $(n-1)$-dimensional subspace of $\ell_1^n$ satisfying $\lambda(Y, \ell_1^{n}) = 2 - \frac{2}{n}$. We shall use Theorem \ref{projekcjal1} and we adapt the notation of it. Obviously $k>2$. If $a_l \geq l-2$ then by the formula on minimal projection we have
$$2 - \frac{2}{n} = \lambda(Y, \ell_1^{n}) = 1 + 2\left ( a_l \beta_l - l \right )^{-1}$$
and thus
$$a_l \beta_l = \frac{2n}{n-2} + l.$$
However, from the Cauchy-Schwarz inequality and $l \leq n$ it follows that
$$a_l \beta_l = \frac{\left ( \sum_{i=1}^{l} f_i \right ) \left ( \sum_{i=1}^{l} f^{-1}_i \right )}{l-2} \geq \frac{l^2}{l-2} = \frac{2l}{l-2} + l \geq \frac{2n}{n-2} + l.$$
In consequence $l=n$ and $1=f_1=f_2=\ldots=f_n$ as the equality holds in the Cauchy-Schwarz inequality.

In the case $a_l < l-2$ we get
$$\left ( \beta_l-f_l^{-1} \right)(l-2) + a_lf^{-1}_l = \frac{2n}{n-2} + l.$$
Since
$$\left ( \beta_l-f_l^{-1} \right)(l-2) + a_lf^{-1}_l \geq \left ( \beta_l-f_l^{-1} \right) a_l + a_lf^{-1}_l = \beta_l a_l,$$
we can apply the Cauchy-Schwarz inequality like before and obtain $1=f_1=f_2 = \ldots = f_n$, which contradicts the assumption $a_l < l-2$. 

To finish the proof of the lemma, it is enough to observe that for $1=f_1=f_2 = \ldots = f_n$ the norm of a minimal projection is equal to $2-\frac{2}{n}$ by the previous theorem. 
\end{proof}

The next lemma follows the original idea of Bohnenblust used in the proof of his theorem.

\begin{lem}
\label{l1}
Let $X$ be an $n$-dimensional normed space and let $Y=\ker f$, where $f \in S_{X^{\star}}$, be its $(n-1)$-dimensional subspace. Suppose that $x_1,\ x_2,\, \ldots,\ x_n$ are unit vectors such that $\lambda(\ell_1^n, Z) \leq R$, where $Z=\ker g$, the functional $g$ is given by the vector $(f(x_1), f(x_2), \ldots f(x_n))$ and $R \geq 1$. Then there exists a projection $P:X \to Y$ such that $||P(x_i)|| \leq R$ for every $i=1, 2, \ldots, n$.
\end{lem}

\begin{proof}
Let $Q: \ell_1^n \to Z$ be a projection of norm at most $R$ and suppose that $Q(x) = x - g(x)r$ for some $r \in \mathbb{R}^n$ satisfying $g(r)=1$. In particular
$$||Q(e_i)|| = |1-f(x_i)r_i| + \sum_{j \neq i}|f(x_i)r_j| \leq R,$$
for every $1 \leq i \leq n$. Consider a linear mapping $P$ of the form $P(x) = x - f(x) \tilde{r}$, where $\tilde{r}=r_1x_1 + r_2x_2 + \ldots + r_nx_n$. It is a projection from $X$ onto $Y$ since $f(\tilde{r})=g(r)=1.$ Moreover
$$||P(x_i)|| = ||(1 - f(x_i)r_i)x_i - \sum_{j \neq i} r_jf(x_i)x_j|| \leq |1-f(x_i)r_i| + \sum_{j \neq i}|f(x_i)r_j| \leq R.$$
This shows that $P$ is a desired projection and the proof is finished. 

\end{proof}

With the preceeding lemmas we are ready to give the proof of Theorem \ref{warunek}.

\emph{Proof of Theorem \ref{warunek}}.
Let us suppose first that the equality $\lambda(Y, X) = 2-\frac{2}{n}$ holds. By Lemma \ref{helly} there are extreme points $x_1, x_2, \ldots, x_n$ of the unit ball of $X$ such that
$$\max \{ ||P(x_1)||, ||P(x_2)||, \ldots, ||P(x_n)|| \} = 2-\frac{2}{n},$$
for every projection $P: X \to Y$. We shall prove that $x_1, x_2, \ldots, x_n$ satisfy the conditions of the theorem.

Since $||P(v)||=||P(-v)||$ for every projection $P: X \to Y$ and every $v \in X$, we can suppose that $f(x_i) \geq 0$ for $1 \leq i \leq n$. By combining Lemmas \ref{l1proj} and \ref{l1} we conclude that $f(x_1)=f(x_2)=\ldots=f(x_n)$. This shows that $x_i$ satisfy the first condition of the theorem.

For the second one, let us suppose that dimension of the subspace $V = \lin \{x_1, x_2, \ldots, x_n \}$ is at most $n-1$. Then dimension of the subspace $Y \cap V$ is at least $\dim V - 1 $ and from the Theorem \ref{twbohnenblust} we know that there exists a projection from $V$ onto $(Y \cap V)$ of the norm not greater than $2 - \frac{2}{\dim V} < 2 - \frac{2}{n}$. This contradicts the choice of $x_i$.

In order to establish the last condition, we shall consider the barycenter $g=\frac{x_1 + x_2 + \ldots + x_n}{n}$. By the triangle inequality it follows that
$$||g-x_i|| =\frac{1}{n}  \left | \left |x_1 + x_2 + \ldots + x_{i-1} - (n-1)x_i + x_{i} + \ldots + x_n \right | \right | \leq \frac{2n-2}{n}=2-\frac{2}{n}.$$
for every $1 \leq i \leq n$. We claim that the equality $||g-x_i||=2-\frac{2}{n}$ holds for every $i$.

Indeed, for the sake of contradiction let us assume that $||g-x_n|| < 2-\frac{2}{n}$. For simplicity let us also denote $A = 2 - \frac{2}{n}$ and $B= \frac{||g-x_n||}{A}<1$. Then we have $ 0 < 2-A < 2-AB$. We can therefore consider 
$$s=\lambda g + (1-\lambda)\frac{x_1 + x_2 + \ldots + x_{n-1}}{n-1}$$
for $\lambda$ satisfying $\frac{2-A}{2-AB} < \lambda < 1$. We claim that $||s-x_i|| < A$ for every $1 \leq i \leq n$. Indeed,
$$||s-x_1|| =  \left | \left |\left ( \frac{1}{n-1} - \frac{\lambda}{n(n-1)} - 1 \right )x_1 + \left ( \frac{1}{n-1} - \frac{\lambda}{n(n-1)} \right ) \sum_{i=2}^{n-1} x_i + \frac{\lambda}{n}x_n \right | \right | \leq$$
$$\left ( 1 -\frac{1}{n-1}  +  \frac{\lambda}{n(n-1)} \right ) + (n-2) \left ( \frac{1}{n-1} - \frac{\lambda}{n(n-1)} \right ) + \frac{\lambda}{n} = \frac{2(n^2-2n+\lambda)}{n(n-1)} < 2 \frac{(n-1)^2}{n(n-1)}=A.$$
Similarly $||s-x_i|| < A$ for $i=2, 3, \ldots, n-1$. Furthermore,
$$||s-x_n||=\left | \left | \lambda(g-x_n) + (1-\lambda) \left( x_n - \frac{\sum_{i=1}^{n-1}x_i}{n-1} \right ) \right | \right | $$
$$\leq \lambda AB + (1-\lambda)2 = \lambda(AB - 2) + 2 < (A-2) + 2 = A.$$
This proves our claim.

Now let us consider the projection $P:X \to Y$ in the direction of $s$, i.e. $P(v)=v - \frac{f(v)}{f(s)}s$. Since $s \in \conv \{x_1, x_2, \ldots x_n\}$ we have $f(s)=f(x_i)$ for every $i$. Hence
$$||P(x_i)|| = ||x_i-s|| < 2 - \frac{2}{n},$$
for every $1 \leq i \leq n$. This contradicts the choice of $x_1, x_2, \ldots, x_n$ and our claim follows.

We have thus proved that the point $g$ is equidistant to every $x_i$ with the distance equal to $2 - \frac{2}{n}$. This means that for every $1 \leq i \leq n$ the point 
$$\frac{1}{2n-2} (x_1 + x_2 + \ldots + x_{i-1} - (n-1)x_i + x_{i+1} + \ldots + x_n),$$
belonging to the convex hull of
$$\{x_1, x_2, \ldots, x_{i-1}, -x_i, x_{i+1}, \ldots, x_n\},$$
lies on the unit sphere of $X$. It clearly implies that the set above is contained in a \textbf{face} of the unit ball. But this face must be a facet as $x_i$ are linearly independent vectors. The third condition now follows from the fact that in the basis of $x_i$ the facet containing $x_1, x_2, \ldots, x_{i-1}, -x_i, x_{i+1}, \ldots, x_n$ is determined by the vector $(1, 1, \ldots, 1 -1, 1, \ldots, 1)$. This completes the proof of the first implication.

Let us now suppose that a subspace $Y$ and extreme points $x_1, x_2, \ldots, x_n$ of the unit ball satisfy all of the conditions. By applying an appropriate linear transformation, without loss of generality we can assume that $x_i=e_i$ is the $i$-th unit vector from the canonical basis of $\mathbb{R}^n$. Then $Y = \ker f$, where $f(v)=v_1+v_2 + \ldots + v_n$. For the sake of contradiction let us further suppose that there exists projection $P: X \to Y$ such that $||P(e_i)|| < 2 - \frac{2}{n}$. As the hyperplane containing $e_i$'s is parallel to $Y$ the projection $P$ acts on them as a translation. In other words, there exists some vector $w=(w_1, w_2, \ldots, w_n)$ such that $P(e_i)=e_i-w$ for every $i$ and $\sum_{i=1}^{n} w_i = 1$. Then
$$2 - \frac{2}{n} > ||P(e_i)|| \geq |w_1 + w_2 + \ldots + w_{i-1} - w_i + w_{i+1} + \ldots + w_n + 1|,$$
for every $1 \leq i \leq n$. Summation of all of these inequalities yields
$$2n - 2 > \sum_{i=1}^{n} |w_1 + w_2 + \ldots + w_{i-1} - w_i + w_{i+1} + \ldots + w_n + 1| \geq  \left | (n-2) \sum_{i=1}^{n} w_i + n \right | = 2n-2.$$
We have obtained a contradiction that finishes the proof of the theorem.

\qed

The characterization above can be stated in purely geometric form, as in the following corollary.

\begin{cor}
\label{charakteryzacjan}
Let $X$ be an $n$-dimensional normed space. Then the following conditions are equivalent:
\begin{itemize}
\item There exists a subspace $Y$ of $X$ such that $\dim Y = n-1$ and $\lambda(Y, X) = 2-\frac{2}{n}$.
\item There exists a linear operator $T: \mathbb{R}^n \to \mathbb{R}^n$ such that $C \subset T(B_X) \subset P$,
\end{itemize}
where $B_X$ is the unit ball of $X$, $C$ is the cross-polytope $\{ x: |x_1| + |x_2| + \ldots + |x_n| \leq 1 \}$ and $P$ is the parallelotope bounded by hyperplanes: $\{x: x_1 + x_2 + \ldots + x_{i-1} - x_i + x_{i+1} + \ldots + x_n = \pm 1\}$ for $i=1, 2, \ldots, n$.
\end{cor}
\begin{proof}
It is enough to take $T$ to be the linear operator such that $T(x_i)=e_i$ for $i=1, 2, \ldots, n$.
\end{proof}

Next corollary of Theorem \ref{warunek} is
\begin{cor}
\label{maxmin}
Let $X$ be an $n$-dimensional normed space which posess an $(n-1)$-dimensional subspace $Y$ such that $\lambda(Y, X) = 2 - \frac{2}{n}$. Then $X$ posess also a two-dimensional subspace $Z$ such that $\lambda(Z, X)=1$.
\end{cor}
\begin{proof}
Let $x_1, x_2, \ldots, x_n$ be like in Theorem \ref{warunek}. From the third condition it follows that for $i \neq j$ the segments connecting pairs $(x_i, x_j), (x_i, -x_j), (-x_i, x_j), (-x_i, -x_i)$ all lie on the unit sphere of $X$. Thus the intersection $\lin\{x_i, x_j\} \cap B_X$ is a parallelogram. It is well-known that $\ell_{\infty}^m$ subspaces are always $1$-complemented (see also Lemma \ref{projekcja}) and the result follows.
\end{proof}

As yet another application of the characterization given in Theorem \ref{warunek} we provide an upper bound for the number of subspaces with maximal relative projection constant in terms of the number of facets of the unit ball.

\begin{twr}
\label{sciany}
Let $X$ be an $n$-dimensional normed space and let $N \geq 0$ be the number of facets of the unit ball of $X$. Then, the number of $(n-1)$-dimensional subspaces $Y \subset X$ such that $\lambda(Y, X) = 2 - \frac{2}{n}$ is not greater than $\binom{N}{n}$.
\end{twr}
\begin{proof}
Let $\mathcal{F}$ be the set of all facets of $B_X$. If subspace $Y$ satisfies $\lambda(Y, X)= 2 - \frac{2}{n}$ then by Theorem \ref{warunek} there exist unit vectors $x_1, x_2, \ldots, x_n$ lying in the hyperplane parallel to $Y$, such that for every $1 \leq i \leq n$ the set
$$\{x_1, x_2, \ldots, x_{i-1}, -x_i, x_{i+1}, \ldots, x_n\}$$
is contained in a different facet of $B_X$. Thus, to every such $Y$ there corresponds a set $F(Y) \subset \mathcal{F}$ of $n$ different facets of $B_X$. We will show that $F$ is an injection.

In this purpose, let us suppose that $n$ facets in $F(Y)$ are determined by the functionals $f_1, f_2, \ldots f_n \in S_{X^{\star}}$ and let $x_1, x_2, \ldots, x_n$ be as before. If $f = \frac{f_1+f_2 + \ldots + f_{n}}{n-2}$, then for every $1 \leq i \leq n$ we have
$$f(x_i) = \frac{(n-1) - 1}{n-2} = 1.$$
This shows that $Y=\ker f$ is uniquely determined by $F(Y)$ and the conclusion follows. Thus, the number of subspaces with the maximal projection costant is at most $\binom{N}{n}$ and the proof is finished.
\end{proof}

As an immediate consequence we have the following

\begin{cor}
\label{wniosek}
An arbitrary $n$-dimensional normed space $X$ posess a subspace $Y$ such that $\lambda(Y, X) < 2 - \frac{2}{n}$.
\end{cor}

We should remark that in the seperable case it is possible that every hyperplane has a maximal possible projection constant. It is well known that if $X$ is a seperable Banach space then $\lambda(Y, X) \leq 2$ for every hyperplane $Y$. And yet, by a result of Franchetti \cite{franchetti3} we have $\lambda(Y, L_1[0, 1]) = 2$ for every hyperplane $Y \subset L_1[0, 1]$.

\section{Three-dimensional case}

In the three-dimensional setting it is possible to establish some stronger results with similar methods. As we have seen in Corollary \ref{charakteryzacjan}, if a three-dimensional space $X$ posess a two-dimensional subspace $Y$ satisfying $\lambda(Y, X) = \frac{4}{3}$, then we can suppose that $C \subset B_X \subset P$, where $C$ is the octahedron $\{ x: |x_1| + |x_2| + |x_3| \leq 1 \}$ and $P$ is the parallelotope with set of vertices: $\{(\pm 1, 0, 0), (0, \pm 1, 0), (0, 0, \pm 1), (1, 1, 1), (-1, -1, -1)\}$. In terms of the norm, we have the following inequalities for an arbitrary vector $x=(x_1, x_2, x_3) \in \mathbb{R}^3$.
$$\max \{ |x_1+x_2-x_3|, |x_1-x_2+x_3|, |-x_1+x_2+x_3| \} \leq ||x|| \leq |x_1|+|x_2|+|x_3|.$$
Note however, that if $x_1\, x_2, x_3$ are not of the same sign then
$$\max \{ |x_1+x_2-x_3|, |x_1-x_2+x_3|, |-x_1+x_2+x_3| \} = |x_1|+|x_2|+|x_3|.$$
Therefore $||x|| = |x_1| + |x_2| + |x_3|$ in such case. This makes the condition much more restrictive on $X$ than in the case of a general dimension. In particular, we are able to determine the maximal possible number of two-dimensional subspaces with the relative projection constant equal to $\frac{4}{3}$. We omit the straightforward proof of the following auxillary lemma.

\begin{lem}
\label{nierownosc}
For arbitrary real numbers $-1 \leq x, y \leq 1$ the inequality $|x| + |y| + |x+y-1| \leq 3$ is true. Moreover, if the equality holds then $x=-1$ or $y-1$ or $x=y=1$.
\end{lem}

\begin{twr}
\label{plaszczyzny}
Let $X$ be a $3$-dimensional normed space. Then, the maximal possible number of $2$-dimensional subspaces $Y \subset X$ such that $\lambda(Y, X) = \frac{4}{3}$ is equal to $4$.
\end{twr}
\begin{proof}
If we take $X = \ell_{\infty}^3$ and $Y=\{x \in \mathbb{R}^3 : c_1x_1 + c_2x_2 + x_3 = 0\}$, where $c_1,\ c_2 \in \{-1, 1\}$ then $\lambda(Y, X) = \frac{4}{3}$. It is therefore enough to prove the upper bound.

Let us start by looking more closely at the case of $X = \ell_{\infty}^3$. Theorem \ref{warunek} implies that every two-dimensional subspace $Y \subset \ell_{\infty}^3$ with the maximal projection constant is parallel to a plane determined by some three of the vertices of the unit cube (not containing any symmetric pair). However, every three vertices lying on one face determine the subspace with the projection constant equal to $1$. Moroever, it is easy to see that every other plane is determined by exactly four $3$-element sets of vertices. This gives exactly four different planes with the maximal projection constant. In this case the statement of the theorem is therefore evident.

Now suppose that $X$ is an arbitrary $3$-dimensional normed space not linearly isometric to the $\ell_{\infty}^3$ and $Y \subset X$ is a subspace with $\lambda(Y, X) = \frac{4}{3}$. Without loss of generality we may assume that $Y = \{x \in \mathbb{R}^3 : x_1 + x_2 + x_3 = 0\}$ and the vectors given by Theorem \ref{warunek} are the vectors $e_1,\ e_2,\ e_3$ from the canonical unit basis of $\mathbb{R}^3$. Let $Z$ be some other subspace of $X$ satisfying $\lambda(Z, X) = \frac{4}{3}$ and denote by $z_1,\ z_2,\ z_3$ the extreme points of the unit ball given by Theorem \ref{warunek}, which are associated with the subspace $Z$. It is enough to show that $\{z_1, z_2, z_3\} = \{\varepsilon_1e_1, \varepsilon_2e_2, \varepsilon_3e_3\}$ for some $\varepsilon_1,\ \varepsilon_2,\ \varepsilon_3 \in \{-1, 1\}$.

Let $P =\{x \in \mathbb{R}^3 : x_1, x_2, x_3 \geq 0\}$ and $-P =\{x \in \mathbb{R}^3 : x_1, x_2, x_3 \leq 0\}$. Note that $z_1,\ z_2,\ z_3 \in (P \cup -P)$. Indeed, $z_i$'s are extreme points of the unit ball and from the remark opening this section it follows that in every part of the coordinate system different from $P$ and $-P$ the unit sphere of $X$ is a triangle with the vertices of the form $\pm e_1,\ \pm e_2,\ \pm e_3$, so these are the only possible extreme points. But these points clearly belong to $P \cup (-P)$.

Without loss of generality we can assume that $z_1,\ z_2 \in P$. Let $z_1 = (a_1, a_2, a_3)$ and $z_2=(b_1, b_2, b_3)$, where $a_i,\ b_i \geq 0$ for $i=1, 2, 3$. Due to the symmetry of the situation we have to consider only two cases: $a_i \geq b_i$ for $i=1, 2, 3$ or $a_1 \geq b_1,\ a_2 \geq b_2$ and $a_3 \leq b_3$. Let us start with the first one.

In this case
$$2 = ||z_1-z_2|| \leq |a_1-b_1| + |a_2 - b_2| + |a_3-b_3| = (a_1+a_2+a_3) - (b_1+b_2+b_3).$$
However,
$$1 = ||z_2|| \leq b_1 + b_2 + b_3.$$
Hence by adding the inequalities
$$|a_1+a_2-a_3| \leq 1, \: |a_1-a_2+a_3| \leq 1, \: |-a_1+a_2+a_3| \leq 1,$$
and using the triangle inequality we obtain that $a_1+a_2+a_3 \leq 3$. Thus $a_1+a_2+a_3=3$, $b_1+b_2+b_3=1$ and the equality holds in all of the estimations. From the equalities
$$|a_1+a_2-a_3| = |a_1-a_2+a_3| = |-a_1+a_2+a_3|  = 1$$
it easily follows that $a_1=a_2=a_3=1$. Therefore, the point $(1, 1, 1)$ belongs to the unit sphere of $X$. This shows that the unit ball of $X$ is the parallelotope with the vertices $\{(\pm 1, 0, 0), (0, \pm 1, 0), (0, 0, \pm 1), (1, 1, 1), (-1, -1, -1)\}$. In particular $X$ is linearly isometric to $\ell_{\infty}^3$, which contradicts our assumption from the beginning of the proof.

Now suppose that $a_1 \geq b_1,\ a_2 \geq b_2,\ a_3 \leq b_3$. Then 
$$2 = ||z_1-z_2|| \leq |a_1-b_1| + |a_2 - b_2| + |a_3-b_3| = (a_1+a_2-a_3) - (b_1+b_2-b_3).$$
However, both of the numbers $|a_1+a_2-a_3|,\ |b_1+b_2-b_3|$ are bounded by $1$. This means that $a_1+a_2-a_3=1$ and $b_1+b_2-b_3=-1$. . Adding the inequalities $|-b_1+b_2+b_3| \leq 1$ and $|b_1-b_2+b_3| \leq 1$ gives us $b_3 \leq 1$. Therefore $b_1=b_2=0$ and $b_3=1$. Note that in fact the absolute value of any coordinate of $z_i$'s is bounded by $1$ by the same argument.

Now we shall incorporate the third point $z_3$ into our reasoning. Let us write $z_3=(c_1, c_2, c_3)$ and suppose that $z_3 \in P$, or in other words that $c_i \geq 0$ for $i=1, 2, 3$. Then we have
$$2 = ||z_2-z_3|| \leq |c_1| + |c_2| + |1-c_3| = c_1 + c_2 - c_3 + 1,$$
so that $c_1 + c_2 - c_3 = 1$. Note that the Theorem \ref{warunek} applied to the plane determined by $z_1, z_2, z_3$ implies the equality $3 = ||z_1-z_2-z_3||$. But on the other hand, we can estimate the norm of this vector using the canonical unit basis of $\mathbb{R}^3$ obtaining the inequality
$$3=||z_1-z_2-z_3|| \leq |a_1 - c_1| + |a_2-c_2| + |a_3-c_3 - 1| = |a_1 - c_1| + |a_2-c_2| + |(a_1-c_1)-(a_2-c_2)-1|.$$
As $0 \leq a_1, a_2, c_1, c_2 \leq 1$ we can apply the Lemma \ref{nierownosc} to $x=a_1-c_1$ and $y=a_2-c_2$ obtaining that $x=y=1$ or $x=-1$ or $y=-1$. In the first case we have that $a_1-c_1=a_2-c_2=1$ which implies that $a_1=a_2=1$ and hence also $a_3=1$. We have thus once again arrived to the case of $\ell_{\infty}^3$ discussed before.

Without loss of generality let us therefore suppose that $a_2-c_2=-1$. Then $a_2=0$ and $c_2=1$. Thus $a_1-a_3=1$ which implies that $a_1=1$ and $a_3=0$. Moreover, $c_3=c_1$. In other words, we have proved that $z_1=e_1$ and $z_3=(c_1, 1, c_1)$. Consequently
$$3=||z_1+z_2-z_3|| \leq |1-c_1| + 1 + |1-c_1| = 3 - 2c_1,$$
which proves that $c_1=0$. This proves our claim in the case $z_3 \in P$.

Suppose now that $c_i \leq 0$ for $i=1, 2, 3$. This time we have
$$2 = ||z_2+z_3| \leq |c_1| + |c_2| + |1+c_3| = 1 - (c_1+c_2-c_3),$$
so that $c_1+c_2-c_3=-1$. Furthemore
$$3=||z_1-z_2+z_3|| \leq |a_1+c_1|+|a_2+c_2|+|a_3+c_3-1| = |a_1+c_1|+|a_2+c_2| + |(a_1+c_1) + (a_2+c_2)-1|.$$
According to our assumptions we have that $0 \leq |a_1+c_1|, |a_2+c_2| \leq 1$ and we can apply Lemma \ref{nierownosc} to $x=a_1+c_1$ and $y=a_2+c_2$. It implies that $a_1+c_1=a_2+c_2=1$ or $a_1+c_1=-1$ or $a_2+c_2=-1$. In the first case it follows that $a_1=a_2=1$ and in consequence $a_3=1$. Similarly like before this means that $X$ is linearly isometric to $\ell_{\infty}^3$, which contradicts our assumption. Therefore, without loss of generality let us assume that $a_2+c_2=-1$. Then it follows that $a_2=0$ and $c_2=-1$. Hence $a_1=1$, $a_3=0$ and $c_1=c_3$.

To finish the proof of the theorem we consider the vector $z_1+z_2-z_3=(1-c_1, -1, 1-c_1)$. On the one hand its norm is equal to $3$ (computing with the respect to the basis of $z_i$). On the other hand, its coordinates are not of the same sign and therefore, computing the norm with the respect to the canonical basis we obtain $3=(1-c_1) + 1 + (1- c_1) = 3-2c_1$, which again gives us $c_1=0$. This completes the proof.



\end{proof}

Rest of this section is devoted to developing the stability version of Corollary \ref{maxmin}. We need the following two lemmas. First of them is a more precise version of Lemma \ref{l1}.

\begin{lem}
\label{l12}
Let $Y=\ker f$ be an $2$-dimensional subspace of the space $\ell_1^3$ and let $0 \leq A < \frac{1}{3}$ be a real number. Suppose that functional $f \neq 0$ is given by the vector $(f_1, f_2, f_3)$, which satisfies
$$1=f_1 \geq f_2 \geq f_3 \geq 0, \qquad f_3 \leq r$$
where 
$$r=\left(\frac{b-\sqrt{b^2-4}}{2} \right)^2  \: \text{ and } \: b=3\sqrt{\frac{1-A}{1-3A}} - 1.$$
Then $\lambda(Y, \ell_1^{3}) \leq \frac{4}{3}-A$.
\end{lem}
\begin{proof}
Note that $b \geq 2$ and in consequence $r \leq 1$. If $f_3=0$ then there is nothing to prove as by Theorem \ref{projekcjal1} we have $\lambda(Y, \ell_1^{3}) = 1 \leq \frac{4}{3}-A$. Let us therefore suppose that $f_3>0$. According to Theorem \ref{projekcjal1} we have
$$\lambda(Y, \ell_1^{n}) = 1+2\left ( (1+f_2+f_3)(1+f_2^{-1}+f_3^{-1}) - 3 \right )^{-1} = 1 + 2\left ( \frac{f_2}{f_3} + \frac{f_3}{f_2} + f_2 + f_3 + \frac{1}{f_2} + \frac{1}{f_3} \right)^{-1}.$$
Our thesis is therefore equivalent to the inequality
$$\frac{f_2}{f_3} + \frac{f_3}{f_2} + f_2 + f_3 + \frac{1}{f_2} + \frac{1}{f_3} \geq \frac{6}{1-3A},$$
under the given conditions.

Consider the function $g(u,v)$ defined as
$$g(u, v) = \frac{u}{v} + \frac{v}{u} + u + v + \frac{1}{u} + \frac{1}{v},$$
for $(u, v) \in \mathbb{R}^2$ satisfying $0 < v \leq u \leq 1$ and $v \leq r$.

By a straightforward calculus we easily obtain that under our assumptions, function $g$ is minimized for $g(1, r)$ or for $g(\sqrt{r}, r)$. However,
$$g(1, r) = 2\left ( r+\frac{1}{r} \right ) + 2 =  \left ( \sqrt{r} + \frac{1}{\sqrt{r}} \right )^2 +  \left ( r+\frac{1}{r} \right) \geq 2 \left ( \sqrt{r} + \frac{1}{\sqrt{r}} \right ) +  \left ( r+\frac{1}{r} \right) = g(\sqrt{r}, r).$$
It is therefore enough to prove that
$$g(\sqrt{r}, r) = 2\left ( \sqrt{r} + \frac{1}{\sqrt{r}} \right ) +  \left ( r+\frac{1}{r} \right ) \geq \frac{6}{1-3A}.$$
If we substitute $t= \sqrt{r} + \frac{1}{\sqrt{r}}$ then it rewrites as
$$t^2 + 2t - 2 \geq \frac{6}{1-3A}$$
or $(t+1)^2 \geq \frac{6}{1-3A}+3$. As $t$ is positive this is equivalent to $t \geq \sqrt{\frac{6}{1-3A}+3} - 1 = b$. After substituting $r=\left(\frac{b-\sqrt{b^2-4}}{2} \right)^2$ we easily check that in fact we have an equality. This finishes the proof of the lemma.

\end{proof}

\begin{lem}
\label{projekcja}
Let $X$ be an $n$-dimensional normed space and suppose that $x, y \in X$ are linearly independent unit vectors satisfying $||x+y||, ||x-y|| \geq 2-A$ for some $0 < A < 1$. Then $\lambda(Y, X) \leq \frac{1}{1-A}$ for $Y=\lin\{x, y\}$.
\end{lem}
\begin{proof}
Without loss of generality we may suppose that $Y = \{(v_1, v_2, 0, \ldots, 0): v_1, v_2 \in \mathbb{R} \}$ and $x=e_1+e_2, y=e_1-e_2$. Then $||e_1+e_2||=||e_1-e_2||=1$ and $||e_1||, ||e_2|| \geq 1-\frac{A}{2}$. We claim that
$$||v|| \leq ||v||_{\infty} \leq \frac{1}{1-A} ||v||,$$
for any $v \in Y$ (where $|| \cdot ||_{\infty}$ denotes the usual supremum norm). In fact, suppose that $v=(v_1, v_2, 0, \ldots, 0)$ for $v_1 \geq v_2 \geq 0$. Then
$$||v||=||v_1e_1 + v_2e_2|| = ||(v_1+v_2)e_1 + v_2(e_1-e_2)|| \geq (v_1+v_2)\left ( 1-\frac{A}{2} \right) - v_2.$$
For a fixed $v_1$ the expression above is a linear function of $v_2$. For $v_2=0$ it is equal to $v_1(1-\frac{A}{2})$ and for $v_2=v_1$ it is equal to $v_1(1-A)$. Thus $||v|| \geq v_1(1-A) = (1-A)||v||_{\infty}$. Furthermore
$$||v|| = \frac{1}{2}||(v_1-v_2)(e_1-e_2) + (v_1+v_2)(e_1+e_2)|| \leq \frac{1}{2} \left ((v_1-v_2) + (v_1+v_2) \right ) = v_1 = ||v||_{\infty}.$$
In the remaining cases the reasoning is analogous. This establishes our claim.

Consider linear functionals $p_1, p_2:Y \to \mathbb{R}$ defined as $p_i(v) =v_i$ for $i=1, 2$. From the previous part it follows that the norms of these functionals are bounded by $\frac{1}{1-A}$. Thus, the Hahn-Banach Theorem gives us extensions $\tilde{p_1}, \tilde{p_2}: X \to \mathbb{R}$ with the norm not exceeding $\frac{1}{1-A}$. Then $P(v)=(\tilde{p_1}(v), \tilde{p_2}(v), 0, \ldots, 0)$ is a desired projection from $X$ onto $Y$ with $||P|| \leq \frac{1}{1-A}$. Indeed
$$||P(v)|| \leq ||P(v)||_{\infty} = \max\{|\tilde{p_1}(v)|, |\tilde{p_2}(v)|\} \leq \frac{1}{1-A} ||v||.$$
This concludes the proof of the lemma.
\end{proof}

Now we are ready to give a stability version of Corollary \ref{maxmin}. We follow a similar reasoning to the one used in the proof of Theorem \ref{warunek}. Before stating the theorem let us introduce the function $\varphi(R)$ as
$$\varphi(R) = \left(\frac{\left ( 3\sqrt{\frac{1-R}{1-3R}} - 1 \right )-\sqrt{\left ( 3\sqrt{\frac{1-R}{1-3R}} - 1 \right )^2-4}}{2} \right)^{-2}-1.$$
The function $\varphi(R)$ will serve us for a quantitative description of the stability. We note that $\varphi$ is continuous and nonnegative on the interval $[0, \frac{1}{3})$ and moreover $\varphi(0)=0$.
\begin{twr}
\label{maxmin3}
Let $X$ be a $3$-dimensional normed space. Suppose that there exists a subspace $Y$ of $X$ such that $\dim Y = 2$ and $\lambda(Y, X) = \frac{4}{3}-R$ for some $R \geq 0$ satisfying $R + \varphi(R) \leq \frac{1}{3}$. Then there exists a $2$-dimensional subspace $Z$ of $X$ such that $\lambda(Z, X) \leq  1 + \frac{9(R+\varphi(R))}{4-12(R+\varphi(R))}$.
\end{twr}

\begin{proof}
Let $X$ and $Y$ be as stated in the theorem and suppose that $Y = \ker f$ for some $f \in S_{X^{\star}}$. As every projection of $X$ onto $Y$ has norm at least $\frac{4}{3}-R$, by Lemma \ref{helly} we conclude that there exist unit vectors $x, y, z \in X$ such that
$$\max \{ ||P(x)||, ||P(y)||, ||P(z)|| \} \geq \frac{4}{3}-R,$$
for every projection $P: X \to Y$. Consider the barycenter $g=\frac{x+y+z}{3}$ and let $C=\frac{6(R+\varphi(R))}{4-3(R+\varphi(R))}$. We claim that among the numbers $||g-x||, ||g-y||, ||g-z||$ there are at least two which are not less than $\frac{4}{3}-C$. 

Suppose otherwise. We can assume that $||g-x||, ||g-y|| < \frac{4}{3}-C.$ Consider $s=\lambda g + (1-\lambda)z$ with $\lambda=\frac{2}{2+C}$. Then $0 < \lambda < 1$ and
$$||s-x|| = ||\lambda(g-x) + (1-\lambda)(z-x)|| < \lambda \left( \frac{4}{3} - C \right) + (1-\lambda)2 = 2 - \frac{2}{3}\lambda - \lambda C = \frac{8}{3C+6}.$$
Similarly $||s-y|| < \frac{8}{3C+6}$. Note also that
$$||g-z|| = \left | \left | \frac{x+y-2z}{3} \right | \right | \leq \frac{1}{3} + \frac{1}{3} + \frac{2}{3} = \frac{4}{3},$$
and hence
$$||s-z|| = \lambda ||g-z|| \leq \frac{4}{3} \lambda = \frac{8}{3C+6}.$$
By a small perturbation of $\lambda$ we can guarantee the strict inequality in the estimation above. In fact, if we replace $\lambda$ by a $\lambda'=\lambda-\varepsilon$ for a sufficiently small $\varepsilon>0$ then for $s'=\lambda' g + (1-\lambda')z$ we still have $||s'-x||, ||s'-y|| < \frac{8}{3C+6}$ but also $||s'-z|| = \frac{4}{3} \lambda' < \frac{8}{3C+6}$. We have thus proved the existence of $s \in \conv \{x, y, z \}$ such that $||s-x||, ||s-y||, ||s-z|| < \frac{8}{3C+6}$.

Since $||P(v)||=||P(-v)||$ for every projection $P: X \to Y$ and every $v \in X$, we can suppose that $f(x), f(y), f(z) \geq 0$. Without loss of generality let us further assume that $f(x) \geq f(y) \geq f(z)$. By Lemma \ref{l12} we have
$$\frac{f(x)}{f(z)} \leq \varphi(R) + 1,$$
as otherwise, by combining Lemmas \ref{l1proj} and \ref{l12} it would be possible to project $x, y, z$ onto $Y$ with a projection of norm smaller than $\frac{4}{3}-R$, contradicting the choice of $x, y, z$. Let $t \in \{x, y, z\}$. Since $s \in \conv\{ x, y, z\}$ it is clear that
$$ \frac{f(t)}{f(s)} - 1 \leq  \frac{f(x)}{f(z)} - 1 \leq \varphi(R),$$
and similarly
$$\frac{f(t)}{f(s)} - 1 \geq \frac{f(z)}{f(x)} - 1 \geq \frac{1}{1+\varphi(R)} - 1 = \frac{-\varphi(R)}{1+\varphi(R)} \geq -\varphi(R).$$
In particular $\left | 1 - \frac{f(t)}{f(s)} \right | \leq \varphi(R)$.

Consider the projection $P:X \to Y$ in the direction of $s$, i.e. $P(v) = v - \frac{f(v)}{f(s)}s$ and let $t \in \{x, y, z\}$. It follows that
$$\left | \left |t - \frac{f(t)}{f(s)}s \right | \right | =  \left | \left |t-s + (1-\frac{f(t)}{f(s)})s) \right | \right| \leq ||t-s|| + \left |1 - \frac{f(t)}{f(s)} \right | ||s||  < \frac{8}{3C+6} + \varphi(R)$$
$$=\frac{8}{3\frac{6(R+\varphi(R))}{4-3(R+\varphi(R))}+6} + \varphi(R)=\frac{8}{\frac{18(R+\varphi(R))}{4-3(R+\varphi(R))}+\frac{24-18(R+\varphi(R))}{4-3(R+\varphi(R))}} + \varphi(R)=\frac{32-24(R+\varphi(R))}{24}$$
$$=\frac{4}{3}-(R+\varphi(R)) + \varphi(R) = \frac{4}{3}-R.$$
We have obtained that $||P(t)|| < \frac{4}{3} - R$ for $t \in \{x, y, z \}$, which contradicts the choice of $x, y, z$. Our claim follows.

Now let us assume that $||g-y||, ||g-z|| \geq \frac{4}{3}-C$. We have
$$4-3C \leq ||3g-3z||=||x+y-2z|| \leq ||x+y|| +2,$$
and therefore $||x+y|| \geq 2-3C$. Moreover
$$4-3C \leq ||3g-3y||=||x-y+2z|| \geq ||x-y||+2,$$
and hence $||x-y|| \geq 2-3C$. From Lemma \ref{projekcja} we conclude that there exists a projection from X onto $\lin\{ x, y \}$ with the norm not exceeding 
$$\frac{1}{1-3C}=\frac{4-3(R+\varphi(R))}{4-21\varphi(R)} = 1 + \frac{18(R+\varphi(R))}{4-21(R+\varphi(R))},$$
as desired.
\end{proof}

As a Corollary we obtain an improvement of the trivial bound in the three-dimensional case of the Problem \ref{bosz}.

\begin{cor}
\label{bosz3}
Every three-dimensional space $X$ posess a subspace $Y$ such that $\lambda(Y, X) < \frac{4}{3} - 0.0007$.
\end{cor}
\begin{proof}
By a numerical calculation one can check that for $R = 0.0007$ we have $R + \varphi(R) < \frac{1}{3}$ and that the inequality 
$$1 + \frac{9(R+\varphi(R))}{4-12(R+\varphi(R))} < \frac{4}{3}-R$$
holds. Therefore if $X$ has a subspace $Y$ with $\lambda(Y, X) > \frac{4}{3} - 0.0007$ then be Theorem \ref{maxmin3} it also has a subspace $Z$ with $\lambda(Z, X) < \frac{4}{3} - 0.0007$.
\end{proof}

\section{Open problems}

In this section we suggest several open questions related to our results.

In Corollary \ref{wniosek} we have established that an arbitrary $n$-dimensional normed space $X$ has a subspace $Y$ with $\lambda(Y, X) < 2 - \frac{2}{n}$. From the standard compactness argument it easily follows that there exists $c>0$ such that every $n$-dimensional normed space $X$ has a subspace $Y$ with $\lambda(Y, X) \leq 2 - \frac{2}{n} - c$. It is natural to ask for determining the best possible constant $c$. In other words, we propose a variation of a Problem \ref{bosz} of Bosznay and Garay.
\begin{problem}
For an integer $n \geq 3$ determine the value of $\sup_{X}  \inf_{Y \subset X} \lambda(Y, X)$, where $X$ is an $n$-dimensional normed space and $Y \subset X$ is a subspace of dimension $n-1$.
\end{problem}
It should be also noted that obtaining any lower bound greater than $1$ would also be significant. Most of the classical normed spaces have always projection of norm $1$, while we need a normed space with every projection of the norm not less than $c$ for some explicit $c>1$. 

It is also natural to ask about improvement of Theorem \ref{sciany}.

\begin{problem}
For an integer $n \geq 3$ determine the maximal possible number of $(n-1)$-dimensional subspaces with the relative projection constant equal to $2-\frac{2}{n}$ that an $n$-dimensional normed space can have. Or at least give some upper bound depending only on $n$ and not on $X$.
\end{problem}

The last question we propose is concerned with Corollary \ref{maxmin}.

\begin{problem}
For an integer $n \geq 3$ determine the maximal possible integer $k$ such that every $n$-dimensional normed space $X$, which posess an $(n-1)$-dimensional subspace $Y$ such that $\lambda(Y, X) = 2 - \frac{2}{n}$, posess also a $k$-dimensional subspace $Z$ satisfying $\lambda(Z, X) = 1$.
\end{problem}

We have established that $k \geq 2$. It is very reasonable to suspect that $k=2$ may be the right answer for this question. Neverthless,  providing a construction of a normed space satisfying such condition would be very interesting.

\end{document}